\documentclass[11pt]{amsart}
\usepackage{amsmath, amssymb, amsthm, geometry, color, enumerate}
\usepackage{url, hyperref}

\def\Z {{\mathbb{Z}}}
\def\R {{\mathbb{R}}}
\def\C {{\mathbb{C}}}
\def\Q {{\mathbb{Q}}}
\def\A {{\mathcal{A}}}
\def\B {{\mathcal{B}}}
\def\D {{\mathcal{D}}}
\def\del{\partial}
\def\delA{\del_{\A}}
\def\delB{\del_{\B}}

\newtheorem{thm}{Theorem}

\newtheorem{theorem}[thm]{Theorem}
\newtheorem{proposition}[thm]{Proposition}

\theoremstyle{remark}
\newtheorem{remark}{Remark}

\begin{document}
\title{Undecidability problems for semifree DG algebras}
\author[Ciprian Manolescu]{Ciprian Manolescu}
\address{Department of Mathematics, Stanford University\\
450 Jane Stanford Way, Building 380, Stanford, CA 94305, USA}
\email {cm5@stanford.edu}

\author[Nick Rozenblyum]{Nick Rozenblyum}
\address {Department of Mathematics, University of Toronto\\ 40 St. George Street, Room 6290, 
Toronto, Ontario M5S 2E4, Canada}
\email {nick.rozenblyum@utoronto.ca}

\begin{abstract}
We prove that the stable tame isomorphism, quasi-isomorphism, and derived Morita equivalence problems for semifree noncommutative differential graded algebras (DGAs) are all undecidable. This resolves half of Problem 5.16 from the K3 Problem List in Low-Dimensional Topology. We present two solutions, both obtained (essentially autonomously) by Gemini Deep Think / Aletheia. 
\end {abstract}

\maketitle
\section{Introduction}

Differential graded algebras play an important role in numerous branches of mathematics, including algebraic geometry, representation theory, homotopy theory, symplectic geometry, and low-dimensional topology. 

Some of the DGAs that appear in practice are graded commutative, and some are not. In particular, in contact geometry, the Chekanov-Eliashberg Legendrian contact DGA \cite{Chekanov, EGH} associates to a Legendrian submanifold a noncommutative DGA that is semifree (i.e., its underlying graded algebra is free). A Legendrian isotopy induces a stable tame isomorphism of the Legendrian contact DGAs. 

It is thus a natural question whether we can tell semifree DGAs apart. There are several variants of this question, because there exist different notions of equivalence for DGAs: stable tame isomorphism, quasi-isomorphism, and derived Morita equivalence (ordered here from the strongest notion to the weakest). We refer to Section~\ref{sec:def} for the exact definitions.

Let $R$ be a nontrivial, unital commutative ring. For the decidability questions below to make sense, we assume that $R$ is Turing computable.\footnote{This includes rings such as $\Z$, $\Q$ and $\Z/n$, but excludes uncountable rings such as  $\R$ or $\C$. One can replace $\R$ and $\C$ with their subfields made of computable numbers. Alternatively, one could work in the BSS model of computation \cite{BSS}, where $\R$ and $\C$ are allowed.} 

\begin{theorem}
\label{thm:main}
Let $\A$ and $\B$ be semifree differential graded algebras over $R$. The following questions are undecidable:

(a) whether $\A$ and $\B$ are stable tame isomorphic;

(b) whether $\A$ and $\B$ are quasi-isomorphic;

(c) whether $\A$ and $\B$ are derived Morita equivalent.
\end{theorem}

This answers one half of Problem 5.16 from \cite{K3}. The other half consists of the analogous problems for semifree graded commutative DGAs, which remain open.

We give two proofs of Theorem~\ref{thm:main}. The first is based on constructing two DGAs that are equivalent (in any of the above senses) if and only if a certain associative algebra is trivial. We then apply the undecidability of the triviality problem for such algebras. The second solution is based on a similar reduction to another known undecidable problem: the triviality of finitely presented groups.

The first proof is shorter than the second. However, the second has the advantage that it explicitly reduces to the  better-known Adyan-Rabin theorem about groups, and that it proves undecidability of a few other related problems. (See Remarks~\ref{rem:1} and \ref{rem:2}.)

\medskip
\textbf{Note on AI use.} The solutions are based on querying {\em Aletheia}, a math research agent based on Gemini Deep Think \cite{Aletheia}. Specifically, the agent produced the first solution when asked about part (b) of Theorem~\ref{thm:main}, and the second when asked about part (a). It also produced a somewhat more complicated solution to part (c), based on the Adams cobar construction, which we do not discuss here. The raw outputs on which this paper is based can be found at: \url{https://github.com/google-deepmind/superhuman/tree/main/aletheia/}

The contributions of the human authors are as follows. They checked the proofs, noticed that either solution applies just as well to all three parts of the problem (by putting together arguments from different AI responses), and added this observation. They also streamlined the exposition, found and added precise references, and wrote the introduction and the background section.   

 In terms of the classification of AI-generated mathematics proposed in \cite[Table 1]{Aletheia}, we judge this paper to be Level 2 and Essentially Autonomous. In retrospect, the solutions are rather simple and an expert in decidability may well have been able to find them with some thought. Nevertheless, they seem to have eluded topologists.

\medskip
\textbf{Acknowledgements.} We thank Google DeepMind and especially Tony Feng, Junehyuk Jung, Quoc V. Le, and Thang Luong for giving us access to the {\em Aletheia} model, and for interesting conversations. We thank Robert Lipshitz for suggesting this problem, and for helpful comments. CM was partially supported by a Simons Investigator Award, the Simons Collaboration Grant on New Structures in Low-Dimensional Topology, and an NSF AIMing grant (DMS-2522743). NR was partially supported by an NSERC Discovery Grant (RGPIN-2025-0696).

\section{Background}
\label{sec:def}
 A differential graded algebra $\A=\oplus_{i \in \Z} \A_i$ over $R$ is a $\Z$-graded unital algebra together with a differential $\del: \A \to \A$ of degree $-1$ satisfying $\del^2=0$ and the graded Leibniz rule. We say that $\A$ is {\em semifree} if, as a graded algebra, $\A$ is isomorphic to the free, noncommutative $R$-algebra on finitely many homogeneous generators $x_i$, i.e., to the tensor algebra $T(x_1, \dots, x_n)$ on the free $R$-module generated by $x_1, \ldots, x_n$. 
 
An {\em elementary automorphism} of a free tensor algebra $T(x_1, \dots, x_n)$ is one that maps exactly one generator $x_i \mapsto \alpha x_i + P$ (where $\alpha \in R^\times$ and $P$ is a homogeneous polynomial exclusively in the remaining generators) and leaves all other generators fixed. A \emph{tame isomorphism} between tensor algebras $T(x_1, \dots, x_n)$ and $T(y_1, \dots, y_n)$ is  the composition  of some elementary automorphisms of $T(x_1, \dots, x_n)$ followed by the isomorphism sending $x_i \mapsto y_i$ or all $i$. 

A \emph{stabilization} of a semifree DGA $\A$ is obtained from $\A$ by adjoining two free generators $e$ and $f$ of degrees $k+1$ and $k$ (for some $k \in \Z$) such that $\del(e) = f$ and $\del(f) = 0$.

Following \cite{Chekanov, ENS}, two semifree DGAs $\A$ and $\B$ are called {\em tame isomorphic} if they are related by a tame isomorphism (of the underlying tensor algebras) that commutes with the differentials.  They are called \emph{stable tame isomorphic} if they become tame isomorphic after applying finitely many stabilizations to each.

A {\em quasi-isomorphism} between DGAs $\A$ and $\B$ is a DG homomorphism $\A \to \B$ inducing an isomorphism on homology; two DGAs are called {\em quasi-isomorphic} if they are related by a zigzag of quasi-isomorphisms. If two semifree DGAs are stable tame isomorphic, they are quasi-isomorphic; see \cite[Corollary 3.11]{ENS}.

Given a DGA $\A$, we denote by $\D(\A)$ the derived category of left DG modules over $\A$. Two semifree DGAs $\A$ and $\B$ are {\em derived Morita equivalent} if $\D(\A)$ and $\D(\B)$ are equivalent as triangulated categories. Quasi-isomorphism implies derived Morita equivalence. We refer to \cite{Keller1, Keller2} for more on the Morita theory for DGAs.

\section{Reduction to algebras}
\label{sec:one}

For the first proof of Theorem~\ref{thm:main}, we use the following result:
\begin{theorem}[Bokut' \cite{Bokut}]
\label{thm:bokut}
Let $A$ be a finitely presented (noncommutative, unital) algebra. The problem of whether $A=0$ is undecidable.
\end{theorem}

In other words, if $$S=\{f_1, \dots, f_m\} \subset R\langle x_1, \dots, x_n \rangle$$ is a finite set of non-commutative polynomials, then it is undecidable whether it generates the unit ideal, i.e., whether $1 \in I := ( f_1, \dots, f_m )$. Indeed, this would correspond to the algebra $$A=R\langle x_1, \dots, x_n \rangle/I$$ being trivial.

Starting from a set $S$ as above (where we may as well assume $n, m \geq 1$) we construct a semifree non-commutative DGA $\A$ equipped with the following generators:
\begin{itemize}
    \item Degree $0$ generators: $x_1, \dots, x_n$ with differential $\delA(x_i) = 0$.
    \item Degree $1$ generators: $r_1, \dots, r_m$ with differential $\delA(r_j) = f_j$.
\end{itemize}
 The degree $0$ module $\A_0$ is the free algebra $R\langle x_1, \dots, x_n \rangle$, and the image of the differential $\delA: \A_1 \to \A_0$ is the two-sided ideal $I$. Consequently, the $0$-th homology is given by $$H_0(\A) = \A_0 / I = A.$$

Next, let $\B$ be the semifree non-commutative DGA with the same generators $x_i$ and $r_j$ as before, with differential
$$\delB(x_i)=0,\ \ \ i=1, \dots, n; $$
$$ \delB(r_j) = 1, \ \ \ j=1, \dots, m.$$
Because $1$ is a boundary, we deduce that $1=0$ in the ring $H_*(\B)$, so $\B$ is acyclic: $H_*(\B)=0$. 

\begin{proposition}
\label{prop:TFAE1}
For the algebras $\A$ and $\B$ described above, the following are equivalent:
\begin{enumerate}[(a)]

\item  $\A$ and $\B$ are stable tame isomorphic;

\item   $\A$ and $\B$ are quasi-isomorphic;

\item  $\A$ and $\B$ are derived Morita equivalent;

\item The algebra $A$ is trivial: $A =0.$
\end{enumerate}
\end{proposition}

\begin{proof}
As mentioned in Section~\ref{sec:def}, we have $(a) \Rightarrow (b) \Rightarrow (c).$ It suffices to prove $(c) \Rightarrow (d)$ and $(d) \Rightarrow (a).$

$(c) \Rightarrow (d)$: Since $\B$ is acyclic, it is quasi-isomorphic (and hence derived Morita equivalent) to $0$. Thus, the derived category $\mathcal{D}(\B)$ is equivalent to the zero category. Since we assumed (c), we deduce that $\D(\A)$ is also equivalent to the zero category. The DGA $\A$ (viewed as a free module over itself) is a compact generator for $\mathcal{D}(\mathcal{A})$ (see \cite{Keller1}), with natural abelian group isomorphisms $$\operatorname{Hom}_{\mathcal{D}(\mathcal{A})}(\mathcal{A}, \mathcal{A}) \cong H_0(\mathcal{A}).$$ Therefore, we must have that  $A=H_0(\A)$ is trivial.

$(d) \Rightarrow (a)$:  Since $H_0(\A)=0$, we have $1=0$ in $H_*(\A)$, so $\A$ is acyclic. By \cite[Corollary 2.2]{DR}, because $\A$ and $\B$ are acyclic and have the same number of generators and relations, they are stable tame isomorphic.
\end{proof}

 Theorem~\ref{thm:main} follows by combining Theorem~\ref{thm:bokut} with Proposition~\ref{prop:TFAE1}.

\section{Reduction to groups}
\label{sec:two}

For the second proof we use the following result, which is arguably better known than Theorem~\ref{thm:bokut}: 

\begin{theorem}[Adyan--Rabin \cite{Adyan, Rabin}]
\label{thm:AR}
Let $G$ be a finitely presented group. The question of whether $G$ is trivial is undecidable.
\end{theorem}

That is, there exists no general algorithm to decide whether an arbitrary finitely presented group $G = \langle g_1, \dots, g_n \mid r_1, \dots, r_m \rangle$ is isomorphic to the trivial group $\{1\}$. 

Starting from such a group presentation, we construct a pair of semifree DGAs $\A$ and $\B$. Let $X = \{g_1, \dots, g_n, g'_1, \dots, g'_n\}$ be a set of $2n$ formal generators, where $g'_i$ acts as the formal inverse of $g_i$. We define elements $f_1, \dots, f_N$ (where $N = m + 2n$) in the free non-commutative algebra $R\langle X \rangle$ corresponding to the defining relations of $G$:
\begin{itemize}
    \item $r_j - 1$ for $j = 1, \dots, m$ (with any occurrence of $g_i^{-1}$ replaced by $g'_i$),
    \item $g_i g'_i - 1$ and $g'_i g_i - 1$ for $i = 1, \dots, n$.
\end{itemize}

Both $\A$ and $\B$ are defined as semifree non-commutative DGAs over $R$ on the same finite set of homogeneous generators:
\begin{itemize}
    \item {Degree 0:} The $2n$ generators $x \in X$.
    \item {Degree 1:} Some $N$ generators $y_j \in Y = \{y_1, \dots, y_N\}$ corresponding to the relations $f_j$.
    \item {Degree 1:} Some $2n$ generators $z_x \in Z = \{z_x \mid x \in X\}$ associated to the elements of $X$.
\end{itemize}

Extended via the graded Leibniz rule, we define the differential $\delA$ for $\A$ on the generators as follows:
\begin{align*}
    \delA(x) &= 0 \quad \text{for all } x \in X, \\
    \delA(y_j) &= f_j \quad \text{for all } j = 1, \dots, N, \\
    \delA(z_x) &= 0 \quad \text{for all } z_x \in Z.
\end{align*}

The differential $\delB$ for $\B$ is defined the same way on $X$ and $Y$, but differently on $Z$:
\begin{align*}
    \delB(x) &= 0 \quad \text{for all } x \in X, \\
    \delB(y_j) &= f_j \quad \text{for all } j = 1, \dots, N, \\
    \delB(z_x) &= x - 1 \quad \text{for all } z_x \in Z.
\end{align*}

\begin{proposition}
\label{prop:TFAE2}
For the algebras $\A$ and $\B$ described above, the following are equivalent:
\begin{enumerate}[(a)]
\item $\A$ and $\B$ are  tame isomorphic;

\item  $\A$ and $\B$ are stable tame isomorphic;

\item   $\A$ and $\B$ are quasi-isomorphic;

\item  $\A$ and $\B$ are derived Morita equivalent;

\item The group $G$ is trivial: $G \cong 1.$
\end{enumerate}
\end{proposition}

\begin{proof} We have $(a) \Rightarrow (b) \Rightarrow (c) \Rightarrow (d).$ It suffices to prove $(d) \Rightarrow (e)$ and $(e) \Rightarrow (a).$

 $(d) \Rightarrow (e)$: Let us compute the $0$-th homology $H_0$ for both DGAs. Because there are no generators in negative degrees, the $0$-cycles consist of the entire free algebra $R\langle X \rangle$. The $0$-th homology is therefore obtained by taking the quotient of this algebra by the two-sided ideal generated by the boundaries of the degree 1 elements.

For $\A$, the boundaries are generated exactly by the relation polynomials $f_j$. Hence, 
\[
H_0(\A) = R\langle X \rangle / \langle f_1, \dots, f_N \rangle \cong R[G],
\]
the group ring of $G$ over $R$.

For $\B$, the boundaries are generated by the ideal $\langle f_1, \dots, f_N, \{x - 1\}_{x \in X} \rangle$. Setting $x = 1$ for all generators  collapses all group elements to the identity, which automatically satisfies the relations $f_j = 1 - 1 = 0$. The algebra therefore retracts to the base ring: 
\[
H_0(\B) = R\langle X \rangle / \langle f_1, \dots, f_N, \{x - 1\}_{x \in X} \rangle \cong R.
\]

If two DGAs are derived Morita equivalent, their Hochschild homologies are isomorphic (see \cite{Keller2}). For any DGA supported in non-negative homological degrees, $\mathit{HH}_0$ is canonically $H_0 / [H_0, H_0]$.

For $\B$, since $H_0(\B) \cong R$, we have $\mathit{HH}_0(\B) \cong R$.

For $\A$, we have $$\mathit{HH}_0(\A) \cong R[G] / [R[G], R[G]] \cong \bigoplus_{[g] \in \text{Conj}(G)} R,$$ which is the free $R$-module generated by the conjugacy classes of $G$. 

Since $\A$ and $\B$ are derived Morita equivalent, the isomorphism $\mathit{HH}_0(\A) \cong \mathit{HH}_0(\B)$ forces the number of conjugacy classes to be exactly $1$. Because the identity element always forms its own conjugacy class, this  implies $G \cong 1$.

$(e) \Rightarrow (a)$: If $G \cong 1$, then in the algebra $R\langle X \rangle / \langle f_1, \dots, f_N \rangle$, every generator $x \in X$ evaluates to $1$. This means the element $x - 1$  resides inside the two-sided ideal generated by $\{f_j\}_{j=1}^N$ in $R\langle X \rangle$. We can therefore express each $x - 1$ as a finite sum:
\[
x - 1 = \sum_{l} p_{x,l} f_{j_{x,l}} q_{x,l}
\]
where $p_{x,l}, q_{x,l} \in \A_0$ are polynomials over the generators in $X$.

We define elements $w_x \in \A_1$ by substituting the generators $y_{j_{x,l}} \in Y$ for the formal relations $f_{j_{x,l}}$:
\[
w_x = \sum_{l} p_{x,l} y_{j_{x,l}} q_{x,l}.
\]
By the graded Leibniz rule, we obtain:
\[
\delA(w_x) = \sum_{l} p_{x,l} \delA(y_{j_{x,l}}) q_{x,l} = \sum_{l} p_{x,l} f_{j_{x,l}} q_{x,l} = x - 1.
\]
Note that $w_x$ is constructed exclusively from the generators $X$ and $Y$; it does not contain any variables from $Z$.

We now define a graded algebra homomorphism $\Phi: \B \to \A$ on the free generators by:
\begin{align*}
    \Phi(x) &= x \quad \text{for all } x \in X, \\
    \Phi(y_j) &= y_j \quad \text{for all } j = 1, \dots, N, \\
    \Phi(z_x) &= z_x + w_x \quad \text{for all } z_x \in Z.
\end{align*}
Because $w_x$ is independent of the $Z$-variables, $\Phi$ factors into a finite sequence of elementary automorphisms (shifting one $z_x$ at a time), making it a tame isomorphism of the underlying free graded algebras. 

We verify that $\Phi$ intertwines with the differentials. On $z_x \in Z$, we have:
\[
\delA(\Phi(z_x)) = \delA(z_x + w_x) = 0 + (x - 1) = x - 1.
\]
This  matches $\Phi(\delB(z_x)) = \Phi(x - 1) = x - 1$. Since $\Phi$ acts as the identity on $X$ and $Y$ (where the differentials of both DGAs agree), $\Phi$ is a tame isomorphism. 
\end{proof}

 Theorem~\ref{thm:main} follows by combining Theorem~\ref{thm:AR} with Proposition~\ref{prop:TFAE2}. 
 
 \begin{remark}
 \label{rem:1}
 The argument in this second proof also shows that the tame isomorphism problem for semifree DGAs is undecidable.
 \end{remark}
 
 \begin{remark}
 \label{rem:2}
 An {\em augmentation} of a DGA $\A$ is a DG map $\epsilon: \A \to R$ (such that $\epsilon(1)=1$). One can ask about the decidability of the three kinds of equivalence questions for augmented DGAs. The proof of Theorem~\ref{thm:main} in Section~\ref{sec:one} does not immediately apply because it involves the zero DGA, which cannot be augmented. The proof in this section does apply though, because  the tame isomorphism in $(e)\Rightarrow (a)$ in Proposition~\ref{prop:TFAE2}  is compatible with the natural augmentations of $\A$ and $\B$ (which send all the $x$’s to $1$ and the other generators to $0$).  
 \end{remark}

\bibliography{biblio}
\bibliographystyle{custom}

\end{document}